\documentclass{amsart}
\usepackage{amsmath}
\usepackage{amsfonts}
\usepackage{amssymb}
\usepackage{amsthm}
\usepackage{xfrac}
\usepackage{cite}
\usepackage{color,hyperref}
\usepackage[table,usenames,dvipsnames]{xcolor}
\usepackage{tikz}
\usetikzlibrary{matrix,arrows,positioning,automata}
\pgfdeclarelayer{background layer}
\pgfsetlayers{background layer,main}

\definecolor{gray9}{gray}{0.9}
\definecolor{gray8}{gray}{0.8}
\definecolor{gray7}{gray}{0.7}
\definecolor{gray6}{gray}{0.6}

\definecolor{darkblue}{rgb}{0.0,0.0,0.3}
\hypersetup{colorlinks,breaklinks,
            linkcolor=darkblue,urlcolor=darkblue,
            anchorcolor=darkblue,citecolor=darkblue}

\usepackage[norelsize,ruled,vlined,linesnumbered]{algorithm2e}

\newcommand{\cS}{{\mathcal S}}

\newcommand{\Z}{\mathbb{Z}}
\newcommand{\N}{\mathbb{N}}
\newcommand{\id}{()}

\newcommand{\B}[1]{\mathbf{#1}}

\newcommand{\PowSet}{{\mathcal P}}
\newcommand{\Ind}{{\mathcal I}}
\newcommand{\IndMax}{{\mathcal I}_{\text{max}}}
\newcommand{\IndGen}{{\mathcal I}_{\langle\rangle}}

\DeclareMathOperator{\Stab}{Stab}

\DeclareMathOperator{\diam}{diam}

\newcommand{\todo}[1]{}

\theoremstyle{plain}
\newtheorem{theorem}{Theorem}[section]
\newtheorem{lemma}[theorem]{Lemma}
\newtheorem{fact}[theorem]{Fact}
\newtheorem{Proposition}[theorem]{Proposition}

\theoremstyle{definition}
\newtheorem{definition}[theorem]{Definition}

\newcommand{\SubSemi}{\textsc{SubSemi}}
\newcommand{\BioGAP}{\textsc{BioGAP}}
\newcommand{\GAP}{\textsc{Gap}}
\newcommand{\Magma}{\textsc{Magma}}

\begin{document}
\title[Independent Generating Sets of Symmetric Groups]{Computational Enumeration of Independent Generating Sets of Finite Symmetric Groups}
\author[A. Egri-Nagy, V. Gebhardt]{Attila Egri-Nagy$^1$ \and Volker Gebhardt$^2$}
\address{$^1$Akita International University, Yuwa, Akita-city
  010-1292, Japan}
\email{attila@egri-nagy.hu}
\address{$^2$Centre for Research in Mathematics, School of Computing, Engineering and Mathematics, Western Sydney University, Locked Bag 1797, Penrith, NSW 2751, Australia}
\email{v.gebhardt@westernsydney.edu.au}
\maketitle

\begin{abstract}We developed computer algebra tools for enumerating
conjugacy classes of independent subsets and generating sets of symmetric groups
up to $n=7$, and carried out an initial analysis of the obtained results.
\end{abstract}

\section{Introduction}

Generating sets are an efficient way of defining algebraic structures, thus it is natural to require them to be as concise as possible.
This minimality condition can be formulated as independence, that none of the generators can be constructed by the remaining elements of the generating set.

Symmetric groups $\cS_n$, the group of all permutations of degree $n$, are central objects in mathematics since all groups can be represented as a subgroup of some permutation group.
Asymptotically, any pair $(x,y)$, $x,y\in\cS_n$ will generate the
alternating group or the symmetric group, with probability
$\frac{3}{4}$ \cite{DixonProb1969}.
Except $n=4$, we can always construct a permutation that will generate $\cS_n$ together with a given permutation \cite{Piccard39,IsaacsZieschang1995}.
In $\cS_n$, the maximal size of an independent set is $n-1$ and in
that case it is a generating set for $\cS_n$
\cite{Whiston2000}. These maximal independent generating sets can be
explicitly constructed \cite{Cameron2002}.


\section{Basic definitions and notation}
The set $\{1,2,\ldots, n\}$ is denoted by $\B{n}$.
The power set of a set $A$ is denoted by $\PowSet(A)$.

A \emph{generating set} of a group $G$ is a subset $A\subseteq G$ such that every element of $G$ can be expressed as a combination  of elements of $A$, denoted by $\langle A \rangle=G$.
Here we do not include the inverses of the generators tacitly, so we generate the group \emph{as a semigroup}.
Alternatively, we say that $\langle A \rangle$, the \emph{subgroup generated by $A$}, is the intersection of all subgroups of $G$ containing $A$.
We consider generating sets essentially the same if they are conjugate.
\begin{definition}
Let $A$ and $B$ subsets of $G$. Then $A$ is \emph{conjugate} to $B$ in $G$, denoted by $A\sim B$ iff there exists $g\in G$ such that $A=g^{-1}Bg$.
\end{definition}
\noindent In $\cS_n$ conjugacy corresponds to simultaneous relabelling of the points of the permutations.

Similar to linear independence in vector spaces, we can define
independence for sets of permutations.

\begin{definition}[Independent set of permutations]
  \label{def:indset}
Let $A$ be a set of permutations. It is \emph{independent} if
$a\notin \langle A\setminus\{a\} \rangle$ for all
$a\in A$.
\end{definition}
A set of permutations may or may not be independent, depending on the actual algebraic structure it is the subset of. The set $\{\id\}$ containing the identity only is not independent in $\cS_n$ considered \emph{as a group}, since by the group axioms the empty set generates the trivial group. However, in $\cS_n$ \emph{as a semigroup}, it is independent.
In this paper we consider independent sets in $\cS_n$ as a semigroup.

Removing elements from an independent set keeps the set independent.
\begin{fact}
Any nonempty subset of an independent set is an independent set.
\end{fact}
\noindent On the contrary, adding elements and keeping the set independent is difficult, even
when we use information from the underlying permutation representation.
Let $A$ be a independent set of permutations acting on $\B{n-1}$, and
$b$ a permutation such that $b\notin\Stab_n$. Then $A\cup\{b\}$ is not guaranteed to be independent.
For example, adding $(14)$ to $\{(12),(23)\}$ yields an independent set, but adding
$(14)$ to $\{(12),(123)\}$ does not since $\{(123),(14)\}$ already
generates $\cS_4$ thus $(12)$ becomes redundant.

\begin{definition}
Let $\Ind(G)$ denote the set of independent subsets of a group $G$. An independent set $A\subseteq G$ is \emph{maximal} if there is no $B\in \Ind(G)$ such that $A\subset B$. The set of all maximal independent sets of $G$ is  denoted by $\IndMax(G)$. The set of independent generating sets of $G$ is denoted by $\IndGen(G)$.
\end{definition}

By definition, $ \PowSet(G)\supseteq \Ind(G) \supseteq \IndMax(G)\supseteq \IndGen(G)$, since an independent generating set of $G$ is necessarily maximal.
Here we study $\IndGen(\cS_n)$, the independent generating sets of the symmetric group.

\section{The folklore}
There are a few standard generating sets for the symmetric group.
For a concise overview and wider context see \cite{KeithConradGeneratingSets}.

\begin{Proposition}[All transpositions]
  \label{thm:alltranspos}
The set of all transpositions $(ij)$, where $i\neq j$ and $i,j\in\B{n}$, generate $\cS_n$.
\end{Proposition}
\begin{proof}
A general element of $\cS_n$ is a product of disjoint cycles, therefore it is enough to show that an arbitrary cycle can be built using transpositions.
Let $(i_1i_2\ldots i_k)$ be a general cycle, then
$$ (i_1i_2\ldots i_k)= (i_{k-1}i_k)(i_{k-2}i_{k-1})\cdots(i_2i_3)(i_1i_2).$$
\end{proof}
\noindent There are $\binom{n}{2}=\frac{n(n-1)}{2}$ transpositions in $\cS_n$, hence it is not an independent generating set.

\begin{Proposition}[Base point transpositions]
  \label{thm:basepoint}
Let $i$ be a chosen base point and the set of transpositions of the form $(ik)$ with $k\in \B{n}$, $k\neq i$ called the \emph{base point transpositions}. This set of $n-1$ transpositions generates $\cS_n$.
\end{Proposition}
\begin{proof}
Without loss of generality we can fix the base point to be 1.
We can express a general transposition by base point transpositions.
$$ (1j)(1i)(1j)=(ij),$$
therefore by Proposition \ref{thm:alltranspos} we can generate $\cS_n$.
\end{proof}

\begin{Proposition}[Transposition chain]
The generating set $\{(12),(2,3),\ldots, (n-1,n)\}$, the \emph{chain of transpositions} generates $\cS_n$.
\end{Proposition}
\begin{proof}
By using the transpositions in the chain we can generate base point transpositions (Proposition \ref{thm:basepoint}):
  $(1i)=(12)(23)\cdots(i-2\ i-1)(i-1\ i)(i-2\ i-1)\cdots(12).$
\end{proof}

\begin{Proposition}[Transposition and a full cycle]
A full cycle on $n$ points $(i_1i_2\ldots i_n)$ and  a transposition of consecutive points $(i_ji_{j+1})$ in the cycle generates $\cS_n$.
\end{Proposition}
\begin{proof}
The transposition can be conjugated by the cycle to produce a transposition chain.
\end{proof}

These generator sets are very `regular' in a sense and thus easier to study, but by no means all the types of generating sets. Therefore a more systematic enumeration is needed.

\section{Algorithmic Enumeration}

\begin{table}
\begin{center}
\begin{tabular}{|l|r|r|r|r|r|r|r|}
\hline
$n$  & 1 & 2 & 3 & 4 & 5 & 6 & 7\\
\hline
$|\Ind(\cS_n)|$  & 1 & 3 & 16 &  413 & 25346 & 6825268 &  750102585\\
\hline
$|\sfrac{\Ind(\cS_n)}{\sim}|$  & 1 & 3  & 6 &   31 &  258 & 10294 & 155305\\
\hline
$|\sfrac{\IndGen(\cS_n)}{\sim}|$  & 1 & 1 & 2 &  14 & 178& 8621 & 126515\\
\hline
\end{tabular}
\end{center}
\caption{The number of independent sets, their conjugacy classes, and also independent generating sets up to conjugation in $\cS_n$ as a semigroup.}
\label{tab:results}
\end{table}

For $n\leq 3$ we can easily find all independent generating sets of $\cS_n$ by
pen and paper calculations.
For the trivial group $\cS_1$, the  only semigroup generating set is  the singleton set containing the identity (while the only independent group generating set is the empty set).
For $\cS_2\cong\Z_2$, the only independent generating set is $\{(12)\}$.
For $\cS_3$, we have $\sfrac{\IndGen(\cS_3)}{\sim}= \{  \{(12),(23)\}, \{(12),(123)\} \} $, the familiar transposition chain, and the transposition and a full cycle generating sets.

We can use brute force enumeration for $\cS_4$ and $\cS_5$ to get the independent sets by simply enumerating all subsets of size $n-1$ and check them for for independence.
The size of the search space is 12950 and $\approx$199.8 million. For $\cS_6$ the brute force algorithm would need to check $\approx$191 trillion cases, therefore we need a better approach.

A customized depth-first graph search algorithm gives us all the independent
sets in the low degree cases.
The nodes are  sets of permutations and edges correspond to adding a new element to a set.
In other words, the search algorithm traverses the Hasse-diagram of the subset lattice.
The search tree can be cut at dependent sets, since those can never be made independent by adding more elements.

Elements of a conjugacy class of subsets can be calculated easily if
we have a representative element, therefore it is enough and more
efficient to collect only representative subsets. For choosing a
representative element of a conjugacy class, we take its minimal element in lexicographic order.
As further optimization, a pre-calculated lookup table is used to
decide which elements of $\cS_n$ may produce the representative for a
given subset, so we do not need to conjugate by all the $n!$
permutations in $\cS_n$.

Alternatively, the method of canonical construction paths
\cite{McKay1998} can also be used to speed up the search. We say that
the canonical way of extending a set is to add a new maximal element
(in lexicographic order) and reject any other extensions. This method
does not require us to check a newly constructed subset against the
database of already visited subsets. However, in some cases it is
faster to decide independence for a subset than checking the canonicity of its construction.

\section{Preliminary Analysis of Results}

By analyzing the results of the enumeration we can make the following observations:
\begin{enumerate}
\item There are \emph{dead ends}, maximal (by inclusion)
  independent sets in $\cS_n$, so all possible extensions lead to
  dependent sets, but they do not generate $\cS_n$.
\item The proportion of two-element independent generating sets is decreasing as $n$ increases.
\item With the comprehensive enumeration we can define the minimum and maximum diameters for $\cS_n$.
\item Most independent generating sets have no symmetries.
\end{enumerate}
Of course, these observations should be understood with the condition `in the low degree cases'.

\subsection{Dead Ends}
All independent generating sets of $\cS_n$ are maximal independent
sets, since simply there exists no other permutation not generated
already.
However, the converse is not true. Maximality of an independent set
does not guarantee that it will generate the symmetric group.
A trivial example of such a set is the singleton $\{()\}$ generating the trivial group.
Since any added finite permutation would generate the identity, there
is no way to make this independent set bigger.
Non-trivial examples were discovered by the analysis of the search-tree.

\begin{fact}
There exists $n\in\N$ such that $\IndMax(\cS_n)\supsetneq\IndGen(\cS_n)$, i.e.~not all maximal
independent sets of permutations generate the symmetric group.
\end{fact}
\begin{proof}
The smallest value is $n=4$:
$$\langle   (1,3), (1,2,3,4) \rangle = D_4,$$
$$\langle  (1,2)(3,4), (1,2,3,4) \rangle = D_4,$$
 $$\langle (1,2)(3,4), (1,3)(2,4) \rangle = \Z_2\times\Z_2.$$
For example, in the first two cases, any permutation $p\in \cS_4\setminus D_4$ generates $\cS_4$ with the 4-cycle $(1,2,3,4)$, forcing the other generator to be redundant.
\end{proof}
The number of dead ends is growing monotonically, see Table \ref{tab:dead-ends}.

\begin{table}
\begin{center}
\begin{tabular}{|l|r|r|r|r|r|r|r|}
\hline
  $n$ &1 &2& 3&4&5 &6&7\\
  \hline
  \#dead ends &1 &1& 1& 4 &19 & 278 & 17591 \\
\hline
\end{tabular}
\end{center}
\caption{Number of dead end independent sets in symmetric groups.}
\label{tab:dead-ends}
\end{table}

\subsection{Size distribution of independent sets}

Asymptotically most pairs of permutations generate the symmetric group \cite{DixonProb1969}, and this may suggest that most independent generating sets would be of size 2. However results show, at least at beginning, that the proportion of of the two-element generating sets is decreasing from $n=3$, see Table \ref{tab:sizes}. 
For $n=8$ we can enumerate the distinct 2-element generating sets, and it is also a relatively small number.
The above percentages are for conjugacy class representatives, but the size distribution of the total number of independent generating sets have similar ratios.

\begin{table}
\begin{center}
\begin{tabular}{|c|r|r|r|r|r|r|}
\hline
  & 2&  3 & 4 & 5 & 6 \\
\hline
$\cS_3$  & 2 (100\%) &  &  &  &  \\
\hline
$\cS_4$  & 5 (35.71\%)& 9 &  &   &   \\
\hline
$\cS_5$  & 31 (17.41\%)& 138 & 9 &  &  \\
\hline
$\cS_6$  & 163 (1.89\%)& 6355  &  2059 & 44 &  \\
\hline
$\cS_7$  & 1576 (1.01\%)& 67078 &   54398 &  3415 &  48   \\
\hline
$\cS_8$  & 21912 &  &    &   &    \\
\hline
\end{tabular}
\end{center}
\caption{Size distribution of independent generating sets of $\cS_n$. Percentages indicate the proportion of 2-element generating sets.}
\label{tab:sizes}
\end{table}

\subsection{Speed of Generating Sets}

We can for instance measure the `speed' of generating sets, i.e.~how quickly they can build the elements of the group.
A generating $S$ for a group $G$ is \emph{slow} if $\diam_S(G)$ is maximal across all generating set. Slow implies independent, since adding a generator only shorten words.

\begin{table}
\begin{center}
\begin{tabular}{|l|r|r|r|r|r|r|r|}
\hline
  $n$ &1 &2& 3&4&5 &6&7\\
  \hline
  $\min_{S}\left(\diam_S(\cS_n)\right)$ &0 &1& 2&4&5 &7 & 8\\
  \hline
  $\max_{S}\left(\diam_S(\cS_n)\right)$ &0 &1& 3&7&14 &18&34\\
\hline
\end{tabular}
\end{center}
\caption{Minimal and maximal diameter values of independent generating sets of symmetric groups.}
\label{tab:speed}
\end{table}
For $n=2$ the only generating set is $\{(12)\}$, therefore it is slow.
For $n=3$, $\diam_{\{(12),(23)\}}(\cS_3)=3$, while $\diam_{\{(123),(23)\}}(\cS_3)=2$, so the transposition chain is the slow one.
For $n=4$ and $n=5$ we have unique slow generating sets:$\{(123),(34)\}$ and $\{(1243),(2,3)(4,5)\}$.
For $n=6$, $\{(23)(456), (12)(34)(56) \}$ and $\{(56), (123456)\}$, so the full cycle and the transposition generating set is slower than the transposition chain.

\subsection{Symmetries of independent generating sets}
\begin{table}
\begin{center}
\begin{tabular}{|c|r|r|r|r|r|r|r|r|r|r|}
\hline
 & $1$& $\Z_2$& $\Z_2\times\Z_2$& $\Z_3$ & $D_4$ & $D_6$&$\cS_3$& $\cS_4$ & $\cS_5$ & $\cS_6$\\
\hline
$\sfrac{\IndGen(\cS_2)}{\sim}$ &    & 1  &  &  &  & & &  & & \\
\hline
$\sfrac{\IndGen(\cS_3)}{\sim}$  & 1 & 1  &  &  &  &  & &  & &\\
\hline
$\sfrac{\IndGen(\cS_4)}{\sim}$ & 8 & 5  &  &  &  & &1   &  & &\\
\hline
$\sfrac{\IndGen(\cS_5)}{\sim}$  & 150 & 25  &  & 1 & &  &1  &1  & & \\
\hline
$\sfrac{\IndGen(\cS_6)}{\sim}$ & 7931 & 645  & 11 & 6 & 4 & &  20&   2 &2 &\\
\hline
$\sfrac{\IndGen(\cS_7)}{\sim}$ & 121426 & 4846  & 78 & 7 & 7 &7 & 134 & 8 &1 & 1\\
\hline
\end{tabular}
\end{center}
\caption{Symmetry groups of independent generating sets of symmetric groups. It is a special property of a set of group elements to have nontrivial symmetries.}
\label{tab:symms}
\end{table}

Most independent generating sets have no symmetries, i.e.~their normalizers are trivial (Table \ref{tab:symms}).
\todo{describe some interesting ones, explain 2 for S6}

\section{Combinatorial Approach: Incremental generating sets}

 For chasing a recursive counting formula of independent generating sets of $\cS_n$, we need a systematic method to build these objects.
This can be achieved for a special class of generating sets.
We look at how to build a generating set for $\cS_n$ if we have one for $\cS_{n-1}$.
\begin{lemma}
Let $A$ be a generating set for $\cS_{n-1}$ naturally acting on
$\mathbf{n-1}$, then $\langle A\cup\{g\}\rangle=\cS_n$ for any $g\in
\cS_n$ naturally acting on $\mathbf{n}$  such that $g\notin \Stab_n$.
\end{lemma}
\begin{proof}
Let $g$ be a transposition of the form $(in)$,
$i\in\mathbf{n-1}$. Then by taking transpositions $(ik)$,
$k\in\mathbf{n-1}\setminus\{i\}$ from $\cS_{n-1}$ we have a
basepoint generating set for $\cS_n$.

\noindent For an arbitrary cycle, let $g=(i_1\ldots i_k n)$.
Then $g(i_k\ldots i_1)=(i_kn)$, i.e.~we multiply $g$
by an element of $\cS_{n-1}$, and we get a transposition of the form
of the previous case.
\noindent For the general case, $g$ consists of several cycles. Let $g=pq$ where $p$ is the cycle containing $n$ and $q$ is the product of some disjoint cycles of $\mathbf{n-1}$. Then since $q^{-1}\in\cS_{n-1}$ we can generate $p=gq^{-1}$, thus we have a single cycle containing $n$.
\end{proof}

\begin{definition}[Incremental generating set] Let $A\in\Ind_{\cS_n}$, then $A$ is \emph{incremental} if $\exists a\in A$ such that $\langle A\setminus \{a\}\rangle\cong\cS_{n-1}$.
\end{definition}

\begin{definition}[Strongly incremental generating set] Let $A\in\Ind_{\cS_n}$, then $A$ is \emph{strongly incremental} if $\forall a\in A$ such that $\langle A\setminus \{a\}\rangle\cong\cS_{n-1}$.
\end{definition}

The base point generating set,  $\{(1,2),(1,3),\ldots,(1,n)\}$, has this property recursively on all of its subsets.

\begin{center}
\begin{tabular}{|c|c|c|c|c|c|c|}
\hline
symmetric groups  & $\cS_1$& $\cS_2$& $\cS_3$& $\cS_4$& $\cS_5$ & $\cS_6$\\
\hline
\#incremental  & 0 & 1 & 2 &  9 & 92& 6907\\
\hline
\end{tabular}
\end{center}

So if we know $|\Ind_{\cS_{n-1}}|$ then we get independent generating sets by adding elements of $\cS_n\setminus\Stab_n$.
Unfortunately not all of them are independent, since the added element may interfere with the irreducubility of the smaller generating set.

\section{Conclusions}

The computational enumeration was done by the \SubSemi~\cite{subsemi}
and \BioGAP~\cite{biogap} (for calculating minimal and maximal
diameters) packages for the \GAP~computer algebra system \cite{GAP4}.
The \SubSemi~package uses multiplication table representations,
therefore it is capable of calculating independent subsets of
semigroups as well, not just groups.

For checking the results, the number of independent subsets of symmetric groups was also
enumerated independently using the \Magma~computer algebra system \cite{magma}.
By a different graph search algorithm and using the permutation group
representation directly, the values of $|\Ind(\cS_n)|$
were calculated up to $n=6$.
Then, using the set of conjugacy class representatives $\sfrac{\Ind(\cS_n)}{\sim}$
constructed by \SubSemi, we enumerated the conjugacy classes, and by
summing
their cardinalities we got the same values of $|\Ind(\cS_n)|$, but by a
different method.

These software tools produced huge amount of data, which in turn
yielded interesting observations and raised new combinatorial
questions providing novel directions for further theoretical work.

\subsection*{Acknowledgement}

Most of this work was carried out at the Centre for Research in
Mathematics, Western Sydney University, and greatly benefited from
discussions with Andrew R. Francis.

\bibliography{../compsemi}
\bibliographystyle{plain}

\appendix
\section{Independent generating sets of $\cS_4$ up to conjugation}
The `well-known' generating sets are highlighted.
\begin{align*}
\sfrac{\IndGen(\cS_4)}{\sim}=\{& \mathbf{ \{ (3,4), (2,3), (1,2) \} },\quad \{ (3,4), (2,3), (1,2)(3,4) \},\\
& \mathbf{ \{(3,4), (2,3), (1,3)\} } ,\quad  \{ (3,4), (2,3), (1,3)(2,4) \},\\
& \{ (3,4), (2,3,4), (1,2)(3,4) \},\\
& \{ (3,4), (2,3,4), (1,3,4) \},\\
& \{ (3,4), (2,3,4), (1,3)(2,4) \},\\
& \{ (3,4), (2,3,4), (1,4,3) \},\\
& \{ (3,4), (2,3,4), (1,4)(2,3) \},\quad \{ (3,4), (1,2,3) \},\\
& \mathbf{ \{ (3,4), (1,2,3,4) \} },\quad \{ (2,3,4), (1,2,3,4) \},\\
& \{ (2,3,4), (1,2,4,3) \}, \quad\{ (1,2,3,4), (1,2,4,3) \}\}
\end{align*}

\end{document}